\documentclass[12pt]{amsart}
\usepackage{todonotes}
\usepackage{amsmath}
\usepackage{paralist}
\usepackage{graphics} %% add this and next lines if pictures should be in esp format
\usepackage{epsfig} %For pictures: screened artwork should be set up with an 85 or 100 line screen
\usepackage{graphicx}
\usepackage{epstopdf}%This is to transfer .eps figure to .pdf figure; please compile your paper using PDFLeTex or PDFTeXify.
\usepackage[colorlinks=true]{hyperref}
\usepackage{pgfplots}

\usepackage[margin=3cm]{geometry}

\hypersetup{urlcolor=blue, citecolor=red}
\newtheorem{theorem}{Theorem}[section]
\newtheorem{corollary}[theorem]{Corollary}
\newtheorem{lemma}[theorem]{Lemma}
\newtheorem*{problem}{Problem}
\theoremstyle{definition}
\newtheorem{definition}[theorem]{Definition}
\newtheorem{example}[theorem]{Example}
\newtheorem{remark}[theorem]{Remark}
\usepackage{pgfplots}
\usepackage{filecontents}

\title[Ordinal entropy determination] %Use the shortened version of the full title
{Entropy determination based on the ordinal structure of a dynamical system}

\author[K. Keller, S. Maksymenko and I. Stolz]{}


\email{keller@math.uni-luebeck.de}
\email{maks@imath.kiev.ua}
\email{stolz@math.uni-luebeck.de}

\thanks{The authors were supported by Marie Curie Actions - International Research Staff Exchange Scheme (IRSES) FP7-People-2011-IRSES. Project number 295164.}

\begin{document}
%\iffalse hier
\maketitle

% Enter the first author's name and address:
\centerline{\scshape Karsten Keller }
\medskip
{\footnotesize
% please put the address of the first author
\centerline{Universit\"{a}t zu L\"{u}beck}
\centerline{Institut f\"{u}r Mathematik, Ratzeburger Allee 160}
\centerline{ L\"{u}beck, 23562, Germany}
} % Do not forget to end the {\footnotesize by the sign }

\medskip
\centerline{\scshape Sergiy Maksymenko}
\medskip
{\footnotesize
% please put the address of the second  and third author
\centerline{Institute of Mathematics of NAS of Ukraine}
\centerline{Tereshchenkivs'ka str.~3}
\centerline{Kyiv 01601, Ukraine}
}

\medskip
\centerline{\scshape Inga Stolz }
\medskip
{\footnotesize
% please put the address of the first author
\centerline{Universit\"{a}t zu L\"{u}beck}
\centerline{Institut f\"{u}r Mathematik, Ratzeburger Allee 160}
\centerline{ L\"{u}beck, 23562, Germany}
}
% Do not forget to end the {\footnotesize by the sign }

\bigskip
% The name of the associate editor will be entered by an editorial staff
% "Communicated by the associate editor name" is not needed for special issue.
\centerline{(Communicated by the associate editor name)}

%The abstract of your paper
\begin{abstract}
The ordinal approach to evaluate time series due to innovative
works of Bandt and Pompe has increasingly established itself
among other techniques of nonlinear time series analysis.
In this paper, we summarize and generalize the theory of determining
the Kolmogorov-Sinai entropy of a measure-preserving dynamical system via
increasing sequences of order generated
partitions of the state space. Our main focus are measuring processes without information loss.
Particularly, we consider the question of the mi\-ni\-mal necessary
number of measurements related to the properties of a given dynamical system.
\end{abstract}

\section{Introduction}\label{sec1}

Since the invention of permutation entropy by Bandt and Pompe \cite{bandt_pompe_2002} and
the proof of its coincidence with Kolmogorov-Sinai entropy for piecewise monotone interval
maps by Bandt et al.~in \cite{bandt_et_al_2002},
there is some increasing interest in considering time series and dynamical systems from the
pure ordinal point of view (see Amig{\'o}, \cite{amigo_2010}). The idea behind this viewpoint is that much information of a system is already contained in
ordinal patterns describing the up and down of its orbits.
This ordinal view can be particularly useful when having physical quantities for which
the statement that a measuring value is larger than another one is well interpretable,
but concrete purely given differences of measuring values are not.
A prominent example is the (indirect) measurement of temperature as the mean kinetic energy
of the particles of a system by a thermometer. One can make statements about what is warmer or colder, but, for example,
the interpretation of an increase by $1^{\circ}\mathrm{C}$ with not knowing the baseline value is complicated.

This paper is generally discussing the Kolmogorov-Sinai entropy from the ordinal viewpoint.
It reviews and particularly extends and generalizes former results given by
Antoniouk et al.~\cite{antoniouk_et_al_2013}, Amig{\'o} \cite{amigo_2012},
Keller \cite{keller_2012}, Keller and Sinn \cite{keller_sinn_2010, keller_sinn_2009} and Amig{\'o} et al.~\cite{amigo_2005}.
Aspects of entropy estimation are touched.\vspace{3mm}

\textit{The framework.}
The basic model of our discussion is a \emph{measure-preserving dynamical system}
$(\Omega,\mathcal{A},\mu,T)$, i.e.~$\Omega$ is a non-empty set whose elements are interpreted
as the states of a system, $\mathcal{A}$ is a sigma-algebra on $\Omega$,
$\mu: \mathcal{A}\to[0,1]$ is a probability measure, and $T:\Omega\hookleftarrow$
is a $\mathcal{A}$-$\mathcal{A}$-measurable $\mu$-preserving map describing the dynamics of the system.
\emph{$\mu$-preserving} means that $\mu(T^{-1}(A))= \mu(A)$ for all $A \in \mathcal{A}$;
the measure $\mu$ is then called $T$-invariant.

We want to have some kind of regularity of $T$ by assuming at least one of the following conditions:
\begin{eqnarray}
T\mbox{ is \emph{ergodic} with respect to }\mu\mbox{, i.e. }\hspace{6cm}\nonumber\\\mu(A)\in\{0,1\}\mbox{ for all }A\in \mathcal{A}\mbox{ with }T^{-1}(A)= A,\label{ergodic}\\
\Omega\mbox{ can be embedded into some compact metrizable space so that }\mathcal{A}=\mathcal{B}(\Omega). \label{nonergodic}
\end{eqnarray}
Here and in the whole paper, $\mathcal{B}(\Omega)$ denotes the \emph{Borel
$\sigma$-algebra} in the case that $\Omega$ is a topological space. As usual, equivalent to $T$ is ergodic with respect to $\mu$,
we say that $\mu$ \emph{is ergodic for} $T$.

Often the states of a system, whatever they are, cannot be accessed directly, but information
on them can be obtained by measurements.
In this paper such measurements are assumed to be given via \emph{observables}
$X_1,X_2,X_3,\ldots $ defined as ${\mathbb R}$-valued random variables on the probability space
$(\Omega ,\mathcal{A},\mu)$. So the measurements are provided by a stochastic process - we say \emph{sequence of observables}
${\bf X}=(X_i)_{i\in {\mathbb N}}$ - whose realization has components
$(X_i(T^{\circ t}(\omega)))_{t\in {\mathbb N}_0}$. Here $X_i(T^{\circ t}(\omega))$ is interpreted
as the $i$-th measured value from the system at time $t$ when starting in state $\omega\in\Omega$.

A priori we have infinitely many observables providing more and more information,
the finite case, however, is included by equality of all $X_i$; $i\geq n$ for some $n\in {\mathbb N}$.
We will write $\mathbf{X}=(X_i)_{i=1}^n$ in the case of finitely many observables and $\mathbf{X}=X$ in the case of only one observable
$X$.

Unless otherwise stated, in the following $(\Omega,\mathcal{A},\mu,T)$ is a measure-preserving dynamical system
and $\mathbf{X}=(X_i)_{i\in {\mathbb N}}$ a sequence of observables.\vspace{3mm}

\textit{Kolmogorov-Sinai entropy.}
In order to recall the Kolmogorov-Sinai entropy,
let $q \in \mathbb{N}$ and $\mathcal{P} = \{P_1,P_2,\ldots,P_q\} \subset \mathcal{A}$ be a finite partition of
$\Omega$, i.e.~$\Omega=\bigcup_{l=1}^q P_l$, $P_l\neq\emptyset$ for $l=1,2,\ldots ,q$, $P_{l_1}\cap P_{l_2}=\emptyset$
for different $l_1,l_2\in\{1,2,\ldots ,q\}$,  and
let $A = \{1,2,\ldots,q\}$ be the corresponding alphabet. Each word $a_1a_2\ldots a_t$ of length $t \in \mathbb{N}$
defines a set
\begin{equation*}
P_{a_1 a_2\ldots a_t}
:=
\{\omega \in \Omega \mid (\omega,T(\omega),\ldots,T^{\circ t-1}(\omega)) \in P_{a_1} \times P_{a_2} \times \ldots \times P_{a_t} \},
\end{equation*}
and the collection of all non-empty sets obtained for such words of length $t$
provides a partition $\mathcal{P}_t \subset \mathcal{A}$ of $\Omega$.
In particular, $\mathcal{P}_1=\mathcal{P}$.

The \emph{entropy rate} of $T$ with respect to an initial partition $\mathcal{P}$ is given by
\begin{equation*}
h_\mu(T,\mathcal{P})
 =
\lim \limits_{t\to\infty} \frac{1}{t} H_{\mu}(\mathcal{P}_t),
\end{equation*}
where $H_\mu(\mathcal{C})$ denotes the \emph{(Shannon) entropy} of a finite partition
$\mathcal{C}=\{C_1,C_2,\ldots,\linebreak C_q\}\subset\mathcal{A}$ of $\Omega$; $q \in \mathbb{N}$, i.e.
\begin{equation*}
H_\mu(\mathcal{C})
=
- \sum_{l=1}^q \mu(C_l) \ln(\mu(C_l))
\end{equation*}
(with $0\ln(0) := 0$), and the \emph{Kolmogorov-Sinai entropy} is defined by
\begin{equation*}
 h_{\mu}^{\mathrm{KS}}(T)
 =
 \sup_{\mathcal{P} \text{ finite partition }} h_{\mu}(T,\mathcal{P}).
\end{equation*}
Although the Kolmogorov-Sinai entropy is well-defined, its determination is not easy.
In some special cases one can find finite partitions already determining it, usually
called generating partitions (see Definition \ref{def:gen_partition}), however, do not exist or are not accessible.
As a substitute, we want to consider special sequences of partitions only depending on the
ordinal structure of a dynamical system.\vspace{3mm}

\textit{Ordinal partitioning.}
For a single observable $X$ on $(\Omega,\mathcal{A},\mu,T)$ and $s,t\in \mathbb{N}_0$ with $s<t$,
consider the bisection
\begin{equation}\label{bisection}
\begin{split}
\mathcal{P}^{X,T}_{s,t}
=
\{&\{\omega \in \Omega \mid X ( T^{\circ s}(\omega)) < X ( T^{\circ t}(\omega))\} ,
\\ &\{\omega \in \Omega \mid X ( T^{\circ s}(\omega)) \geq X ( T^{\circ t}(\omega))\}\}
\end{split}
\end{equation}
of $\Omega$ and, for observables $X_1,X_2,\ldots ,X_n$ on $(\Omega,\mathcal{A},\mu,T)$ and
$d,n\in {\mathbb N}$, the partition
\begin{equation}\label{ppart}
\mathcal{P}_d^{(X_i)_{i=1}^n,T}
=
\bigvee_{i=1}^n\ \bigvee_{0\leq s<t\leq d}\mathcal{P}_{s,t}^{X_i,T},
\end{equation}
i.e.~the coarsest partition refining all bisections $\mathcal{P}_{s,t}^{X_i,T}$; $i=1,2,\ldots n$, $0\leq s<t\leq d$.
(If one of the sets of the right hand side of \eqref{bisection} is empty, $\mathcal{P}^{X,T}_{s,t}$ is considered to consist
of only one set.)

The partition $\mathcal{P}_d^{(X_i)_{i=1}^n,T}$ is called \emph{ordinal partition} of
\emph{order} $d$ \emph{associated to $(X_i)_{i=1}^n$}. By definition
its parts contain all states with equal ordinal measurement structure for an initial orbit part.\vspace{3mm}

\textit{A central statement.}
Clearly, in order to preserve information of the given system, the observables should separate orbits
of the system in a certain sense. In order to give a precise description, let in the following
$\sigma((\mathbf{X} \circ T^{\circ t})_{t\in {\mathbb N}_0})$ be the $\sigma$-algebra
generated by all random variables $X_i\circ T^{\circ t}$; $i\in {\mathbb N}$, $t\in {\mathbb N}_0$ and
write $\mathcal{F}\overset{\mu}{\supset}\mathcal{G}$ if for each $G\in \mathcal{G}$ there exists
some $F\in \mathcal{F}$ with $\mu(F\,\Delta\,G)=0$.

The following generalization of a statement in Antoniouk et al.~\cite{antoniouk_et_al_2013} says that
if there is no information loss by measuring with observables, all information is preserved also by
only considering measurements from the ordinal viewpoint.

\begin{theorem}\label{main}
Let $(\Omega, \mathcal{A},\mu, T$) be a measure-preserving dynamical system and
$\mathbf{X} = (X_i)_{i \in \mathbb{N}}$ be a sequence of observables such that
$\sigma((\mathbf{X} \circ T^{\circ t})_{t\in\mathbb{N}_0}) \overset{\mu}{\supset} \mathcal{A}$.
Assume that \eqref{ergodic} or \eqref{nonergodic} holds. Then
\begin{equation}\label{equ:hKS_lim_perm}
h_\mu^{\mathrm{KS}}(T)
=
\lim\limits_{d,n\to\infty} h_\mu(T,\mathcal{P}_d^{(X_i)_{i=1}^n,T})
=
\sup_{d,n\in {\mathbb N}} h_\mu(T,\mathcal{P}_d^{(X_i)_{i=1}^n,T}).
\end{equation}
\end{theorem}
When Bandt and Pompe \cite{bandt_pompe_2002} invented the permutation entropy, they considered
one-dimensional systems with coincidence of states and measurements. This fits into the given general
approach as follows: $\Omega$ is a Borel subset of ${\mathbb R}$ and only one observable is considered
to be the \emph{identity map} $\mathrm{id}$ from $\Omega$ into ${\mathbb R}$.
In this situation the assumptions of Theorem \ref{main} are satisfied and so it holds
\begin{equation*}
h_\mu^{\mathrm{KS}}(T)
=
\lim\limits_{d\to\infty} h_\mu(T,\mathcal{P}_d^{\mathrm{id},T})=\sup_{d\in {\mathbb N}} h_\mu(T,\mathcal{P}_d^{\mathrm{id},T})
\end{equation*}
(compare \cite{keller_sinn_2010, keller_sinn_2009}).\vspace{3mm}

\textit{Structure of the paper.} The paper is organized as follows.
In Section \ref{sec2} we provide a proof of Theorem \ref{main} on the basis of Antoniouk et al.~\cite{antoniouk_et_al_2013}.
We, moreover, discuss this statement from different perspectives in Section \ref{sec3} by presenting its modifications and variants. %changed 19.01.2015 11:51
Section \ref{sec4} is devoted to the concept of permutation entropy, in particular to the two different approaches to it given by Bandt et al.~in \cite{bandt_et_al_2002}
and Amig{\'o} et al.~in \cite{amigo_2005}, respectively, and to its relation to the Kolmogorov-Sinai entropy.
The ordinal approach to dynamical systems opens new perspectives to the estimation of system complexity. Advantages and limitations of %changed 19.01.2015 11:53
this approach are discussed in Section \ref{sec5}. The natural question of how many observables are necessary for satisfying the assumptions 
of Theorem \ref{main} is in the focus of Section \ref{sec6}. The corresponding discussion is
strongly %changed 19.01.2015 11:53
related to Takens' %changed 19.01.2015 11:54
delay embedding and similar ideas (see Takens \cite{takens_81} and Sauer \cite{sauer_et_al_1991}).

\section{Kolmogorov-Sinai entropy from the ordinal viewpoint}\label{sec2}
This section is devoted to the proof of Theorem \ref{main}.\vspace{3mm}

\textit{Preliminaries.}
In the following we write $\mathcal{F} \overset{\mu}{=} \mathcal{G} $ if
$\mathcal{F}\overset{\mu}{\supset}\mathcal{G}$ and $\mathcal{F}\overset{\mu}{\subset}\mathcal{G}$,
and denote by $\mathbf{1}_A$ the \emph{indicator function} of a subset $A \subset \Omega$.
Moreover $\sigma (\diamondsuit)$ denotes the $\sigma$-algebra generated by a set $\diamondsuit$ of subsets of $\Omega$,
by a sequence or double sequence $\diamondsuit$ of sets of subsets of $\Omega$, or by a random variable $\diamondsuit$ on $\Omega$.\vspace{3mm}

Given two finite
partitions $\mathcal{C},\mathcal{D}\subset \mathcal{A}$ of $\Omega$, we write $\mathcal{C}\prec\mathcal{D}$
if $\mathcal{D}$ is \emph{finer} than $\mathcal{C}$ or, equivalently, if $\mathcal{C}$ is \emph{coarser}
than $\mathcal{D}$, that is, each element $C \in \mathcal{C}$ is a finite union of some elements of $\mathcal{D}$.
Note that $\prec$ on the set of finite partitions of $\Omega$ contained in $\mathcal{A}$ is a partial order.

The \emph{join} $\bigvee_{r=1}^m \mathcal{C}_r$ of $m \in \mathbb{N}$ finite partitions
$\mathcal{C}_r = \{C_r^{(1)},C_r^{(2)},\ldots,C_r^{(\vert \mathcal{C}_r \vert)}\} \subset \mathcal{A}$ of $\Omega$ with $r=1,2,\ldots, m$
is the coarsest partition refining all
$\mathcal{C}_r$; $r=1,2,\ldots, m$, i.e.
\begin{equation*}
\bigvee_{r=1}^m \mathcal{C}_r
=
\{\bigcap_{r=1}^m C_r^{(l_r)}\neq\emptyset \mid l_r\in\{1,2,\ldots,\vert \mathcal{C}_r\vert\}\mbox{ for }r=1,2,\ldots ,m\}.
\end{equation*}
For an observable $Y$ on $(\Omega,\mathcal{A},\mu,T)$ we consider the finite partitions
\begin{equation*}
\mathcal{P}_d^{Y,T}:=\bigvee_{0\leq s<t\leq d}\mathcal{P}_{s,t}^{Y,T}\mbox{ and }\widetilde{\mathcal{P}}_d^{Y,T}:=\bigvee_{0<t\leq d}\mathcal{P}_{0,t}^{Y,T}
\end{equation*}
(compare \eqref{bisection}) for $d\in {\mathbb N}$ and the $\sigma$-algebras $\Sigma^{Y,T}$ and $\widetilde{\Sigma}^{Y,T}$ generated from all $\mathcal{P}_d^{Y,T}$
and $\widetilde{\mathcal{P}}_d^{Y,T}$; $d\in {\mathbb N}$, respectively.

Besides $\mathcal{P}_d^{(X_i)_{i=1}^n,T}=\bigvee_{i=1}^n \mathcal{P}_d^{X_i,T}$ (compare \eqref{ppart}), for $d,n\in {\mathbb N}$ we are interested in the finite partitions
\begin{equation}\label{less}
\widetilde{\mathcal{P}}_d^{(X_i)_{i=1}^n,T}
:=
\bigvee_{i = 1}^n\widetilde{\mathcal{P}}_d^{X_i,T}.
\end{equation}
Furthermore, we need the following $\sigma$-algebras associated to these partitions:

\begin{equation*}
\Sigma^{\mathbf{X},T}
:=
\sigma\left(\left(\mathcal{P}_d^{(X_i)_{i=1}^n,T}\right)_{d,n\in\mathbb{N}}\right)
=
\sigma\left(\left(\Sigma^{X_i,T}\right)_{i\in\mathbb{N}}\right)
\end{equation*}
and
\begin{equation*}
\widetilde{\Sigma}^{\mathbf{X},T}
:=
\sigma\left( \left(\widetilde{\mathcal{P}}_d^{(X_i)_{i=1}^n,T} \right)_{d,n \in \mathbb{N} } \right)
=
\sigma\left(\left(\widetilde{\Sigma}^{X_i,T}\right)_{i\in\mathbb{N}}\right).
\end{equation*}\vspace{3mm}

\textit{The proof.} Although we consider dynamical systems equipped with infinitely many observables,
we can follow closely the argumentation in the paper Antoniouk et al.~\cite{antoniouk_et_al_2013}. So let us first recall or modify
those statements of that paper used in our proof.

\begin{lemma}\cite[Lemma 3.2]{antoniouk_et_al_2013}\label{distribution}
Let $F:\mathbb{R}\to[0,1]$ be the distribution function of an observable
$X$, that is $F(a) = \mu(\{\omega\in \Omega \mid X(\omega)\leq a\})$ for all $a \in \mathbb{R}$.
Then
\begin{equation*}
\sigma(F \circ X)
\overset{\mu}{=}
\sigma(X).
\end{equation*}
\end{lemma}

\begin{lemma}\cite[Lemma 3.3]{antoniouk_et_al_2013}\label{convergence}
Let $T: \Omega \hookleftarrow$ be an ergodic map and let
$I_d:\Omega\to\mathbb{R}$ be defined by
$I_d(\omega) := \sum_{t= 1}^{d} \mathbf{1}_{\{X ( T^{\circ t}(\omega)) \leq X(\omega)\}}$
for all $d \in \mathbb{N}$ and $\omega \in \Omega$.
Then
\begin{equation*}
F(X(\omega))
=
\lim\limits_{d\to\infty} \frac{I_d(\omega)}{d} \text{ for a.e.~} \omega \in \Omega.
\end{equation*}
\end{lemma}
By very slight modifications we can extend \cite[Corollary 3.4 and Corollary 3.5]{antoniouk_et_al_2013}
to countably many observables:
\begin{corollary}\label{inclusions}
Let $T: \Omega \hookleftarrow$ be an ergodic map.
Then
\begin{equation*}
\sigma(\mathbf{X})
\overset{\mu}{\subset}
\widetilde{\Sigma}^{\mathbf{X},T}
\subset\Sigma^{\mathbf{X},T}.
\end{equation*}
\end{corollary}

\begin{proof}
Compare to \cite[Corollary 3.4]{antoniouk_et_al_2013}. The $\sigma$-algebra $\widetilde{\Sigma}^{\mathbf{X},T}$
is generated by the $\sigma$-algebras
$\widetilde{\Sigma}^{X_i,T}
:=
\sigma((\bigvee_{0<t\leq d}\mathcal{P}_t^{X_i,T})_{d \in \mathbb{N}})$; $i \in \mathbb{N}$.
Therefore by $\sigma(X_i) \overset{\mu}{\subset} \widetilde{\Sigma}^{X_i,T}$ for all $i \in \mathbb{N}$
it follows the assumption. This is true since
$\frac{I_d}{d}:\Omega\to[0,1]$ is $\widetilde{\Sigma}^{X,T}$-$\mathcal{B}([0,1])$-measurable
for all $d \in \mathbb{N}$ and hence so is $F \circ X$ and $X$ by Lemma \ref{distribution}
and Lemma \ref{convergence}. The inclusion $\widetilde{\Sigma}^{\mathbf{X},T}
\subset\Sigma^{\mathbf{X},T}$ is given by construction (compare \eqref{ppart} and \eqref{less}).
\end{proof}

\begin{corollary} \label{inclusion}
Let $T: \Omega \hookleftarrow$ be an ergodic map. Then
\begin{equation*}
\sigma((\mathbf{X} \circ T^{\circ t})_{t \in \mathbb{N}_0})
\overset{\mu}{\subset}
\Sigma^{\mathbf{X},T}.
\end{equation*}
\end{corollary}
\begin{proof}
For fixed $n\in {\mathbb N}$, in \cite[Proof of Corollary 3.5]{antoniouk_et_al_2013} it is shown that
\begin{equation}\label{puneq}
\mathcal{P}_d^{X_i \circ T,T} \prec \mathcal{P}_{d+1}^{X_i,T}\mbox{ for all }d \in \mathbb{N}\mbox{ and }i=1,2,\ldots,n
\end{equation}
implying
\begin{equation}\label{sigmauneq}
\Sigma^{X_i\circ T^{\circ t},T}\subset \Sigma^{X_i,T}\mbox{ for all }i=1,2,\ldots,n\mbox{ and }t\in\mathbb{N}_0.
\end{equation}
Moreover, Corollary \ref{inclusions} gives
\begin{equation}\label{suneq}
\sigma(\mathbf{X}\circ T^{\circ t}) \overset{\mu}{\subset} \Sigma^{\mathbf{X}\circ T^{\circ t},T}
\mbox{ for all }t \in \mathbb{N}_0.
\end{equation}
Consequently,
$\sigma((\mathbf{X} \circ T^{\circ t})_{t \in \mathbb{N}_0}) \overset{\mu}{\subset} \Sigma^{\mathbf{X},T}$.
\end{proof}

\begin{lemma}\label{increasing}
$(\mathcal{P}_d^{(X_i)_{i=1}^n,T})_{d,n \in \mathbb{N}}$ is an increasing sequence in $n$ for fixed $d$,
and for fixed $n$ it is an increasing sequence in $d$.

In particular,
$(\mathcal{P}_{d_j}^{(X_i)_{i=1}^{n_j},T})_{d_j,n_j \in \mathbb{N}}$ is an increasing sequence in $j$ if
$(d_j)_{j \in \mathbb{N}}$ and $(n_j)_{j \in \mathbb{N}}$ are increasing sequences in $\mathbb{N}$.
\end{lemma}

\begin{proof}
Given $d,n \in \mathbb{N}$, it holds
\begin{equation*}
\begin{split}
\mathcal{P}_d^{(X_i)_{i=1}^n,T}&=\bigvee_{i=1}^n\ \bigvee_{0\leq s<t\leq d}\mathcal{P}_{s,t}^{X_i,T},\\
\mathcal{P}_{d+1}^{(X_i)_{i=1}^n,T}&=\bigvee_{i=1}^n\ \bigvee_{0\leq s<t\leq d+1}\mathcal{P}_{s,t}^{X_i,T},\\
\mathcal{P}_d^{(X_i)_{i=1}^{n+1},T}&=\bigvee_{i=1}^{n+1}\ \bigvee_{0\leq s<t\leq d}\mathcal{P}_{s,t}^{X_i,T},
\end{split}
\end{equation*}
implying $\mathcal{P}_d^{(X_i)_{i=1}^n,T}\prec \mathcal{P}_{d+1}^{(X_i)_{i=1}^n,T},\mathcal{P}_d^{(X_i)_{i=1}^{n+1},T}$ and
so the above statements.
\end{proof}
For completing the proof of Theorem \ref{main},
we apply the following statement (see Walters \cite[Theorem 4.22]{walters_2000}):

\begin{lemma}\label{walters}
For a sequence
$(\mathcal{C}_d)_{d \in \mathbb{N}}$ of finite partitions
$\mathcal{C}_d\in \mathcal{A}$ of $\Omega$ increasing
with respect to $\prec$ and satisfying $\sigma(( \mathcal{C}_d)_{d \in \mathbb{N}})\overset{\mu}{\supset}\mathcal{A}$,
it holds
\begin{equation*}
h_{\mu}^{\mathrm{KS}}(T)
=\lim\limits_{d\to\infty} h_{\mu}(T,\mathcal{C}_d).
\end{equation*}
\end{lemma}
First suppose that $T$ is an ergodic map. Then under the assumptions of Theorem \ref{main} and %Lemma \ref{lem:inclusion_ord_struc_observable} and
by Corollary \ref{inclusion} it holds
$\mathcal{A} \overset{\mu}{\subset} \sigma((\mathbf{X} \circ T^{\circ t})_{t \in \mathbb{N}_0}) \overset{\mu}{\subset} \Sigma^{\mathbf{X},T}$. Since
by Lemma \ref{increasing}
$(\mathcal{P}_{d_j}^{(X_i)_{i=1}^{n_j},T})_{d_j,n_j \in \mathbb{N}}$
is an increasing sequence in $j$ with respect to $\prec$ for increasing sequences
$(d_j)_{j \in \mathbb{N}}$ and $(n_j)_{j \in \mathbb{N}}$ in $\mathbb{N}$,
the assertion of Theorem \ref{main} follows from Lemma \ref{walters}.

In the non-ergodic case the ergodic decomposition
theorem is consulted. For a thorough treatment we refer the reader to
Einsiedler and Ward \cite{einsiedler_2010} and Einsiedler et al.~\cite{einsiedler_et_al_2015}.
In particular, the ergodic decomposition theorem claims that under certain conditions any $T$-invariant measure $\mu$
can be decomposed into ergodic components and subsequently the entropy rate as well as the Kolmogorov-Sinai
entropy of $T$ with respect to $\mu$ can be written as the integral of the entropies
with respect to the decomposition.

In order to complete the proof of Theorem \ref{main},
we apply the following statement (see Einsiedler and Ward \cite[Theorem 6.2]{einsiedler_2010},
Einsiedler et al.~\cite[Theorem 5.27]{einsiedler_et_al_2015}
and Keller and Sinn \cite{keller_sinn_2010} for the case of a non-invertible $T$):

\begin{theorem}\label{ergodicdecomposition}
Let $(\Omega,\mathcal{A},\mu,T)$ be a measure-preserving dynamical system satisfying \eqref{nonergodic}. Then there
exists a probability space $(\Omega^\ast,\mathcal{A}^\ast,\nu)$ and a map
$\omega^\ast \mapsto \mu_{\omega^\ast}$ associating to each $\omega^\ast \in \Omega^\ast$
a probability measure $\mu_{\omega^\ast}$ on $(\Omega,\mathcal{A})$
such that the following is valid:

$\Omega^\ast$ can be embedded into some compact metrizable space so that $\mathcal{A}=\mathcal{B}(\Omega^\ast)$, the map
$\omega^\ast \in \Omega^\ast \to \int_{\Omega^\ast} f \,\mathbf{d} \mu_{\omega^\ast}$
is $\mathcal{A^\ast}$-$\mathcal{B}(\mathbb{R})$-measurable for every essentially bounded measurable function $f:\Omega \to \mathbb{R}$,
the measure $\mu_{\omega^\ast}$ is ergodic %changed 19.01.2015 11:55
$T$-invariant for $\nu$-a.e.~$\omega^\ast \in \Omega^\ast$,
and
\begin{equation*}
\mu
=
\int_{\Omega^*} \mu_{\omega^\ast} \,\mathbf{d}\:\!\nu(\omega^\ast).
\end{equation*}
Moreover, it holds
\begin{equation}
h^{\mathrm{KS}}_\mu(T)
=
\int_{\Omega^*} h^{\mathrm{KS}}_{\mu_{\omega^\ast}}(T)\,\mathbf{d}\:\!\nu(\omega^\ast)
\label{decomposition}
\end{equation}
and
\begin{equation}
h_\mu(T,\mathcal{P})
=
\int_{\Omega^*} h_{\mu_{\omega^\ast}}(T,\mathcal{P})\,\mathbf{d}\:\!\nu(\omega^\ast)
\text{ for each finite partition } \mathcal{P} \subset\mathcal{A} \text{ of } \Omega.
\label{entropy}
\end{equation}
\end{theorem}
Altogether we obtain
\begin{eqnarray*}
h^{\mathrm{KS}}_{\mu}(T)
&\overset{\text{\eqref{decomposition}}}{=}&
\int_{\Omega^*} h^{\mathrm{KS}}_{\mu_{\omega^\ast}}(T)\,\mathbf{d}\:\!\nu(\omega^\ast)\\
&\overset{\text{Theorem \ref{main}}}{\underset{\text{ergodic case}}{=}}&
\int_{\Omega^*} \lim\limits_{j \to \infty}h_{\mu_{\omega^\ast}}(T,\mathcal{P}_{d_j}^{(X_i)_{i=1}^{n_j},T})\,\mathbf{d}\:\!\nu(\omega^\ast)\\
&\overset{\text{monotone }}{\underset{\text{convergence}}{=}}&
\lim\limits_{j \to \infty} \int_{\Omega^*} h_{\mu_{\omega^\ast}}(T,\mathcal{P}_{d_j}^{(X_i)_{i=1}^{n_j},T})\,\mathbf{d}\:\!\nu(\omega^\ast)\\
&\overset{\text{\eqref{entropy}}}{=}&
\lim\limits_{j \to \infty} h_{\mu}(T,\mathcal{P}_{d_j}^{(X_i)_{i=1}^{n_j},T}).
\end{eqnarray*}
Here $(n_j)_{j\in {\mathbb N}}$ and  $(d_j)_{j\in {\mathbb N}}$ are strictly increasing sequences of natural numbers.

\section{Modifications and conseqences of Theorem \ref{main}.}\label{sec3}
We want to have a closer look at Theorem \ref{main}. For this
recall that $\mathbf{X} \circ T^{\circ t}$ can be interpreted as a measurement of a system at time $t$. As discussed in Section \ref{sec1},
there is no information loss when taking a pure ordinal viewpoint in the case that these measurements have `separating properties'.\vspace{3mm}

\textit{Less comparisons.}
The main Theorem \ref{main} can be given in a relaxed version if the considered observables provide
a `separation' from the outset (compare also \cite{keller_sinn_2009, keller_et_al_2007}).
In order to determine the Kolmogorov-Sinai entropy,
this means, in the case of `separating' original observables, one does not
need all comparisons between the elements of an orbit but only comparisons between points and their iterates.
\begin{theorem}\label{main2}
Let $(\Omega, \mathcal{A},\mu, T$) be a measure-preserving dynamical system and
$\mathbf{X} = (X_i)_{i \in \mathbb{N}}$ be a sequence of observables such
that $\sigma(\mathbf{X}) \overset{\mu}{\supset} \mathcal{A}$. Assume that \eqref{ergodic} or \eqref{nonergodic} holds.
Then
\begin{equation*}
h_\mu^{\mathrm{KS}}(T)
=
\lim\limits_{d,n\to\infty} h_\mu(T,\widetilde{\mathcal{P}}_d^{(X_i)_{i=1}^n,T})
=
\sup_{d,n\in {\mathbb N}} h_\mu(T,\widetilde{\mathcal{P}}_d^{(X_i)_{i=1}^n,T}).
\end{equation*}
\end{theorem}

For an ergodic map $T$ we have that
$ \mathcal{A} \overset{\mu}{\subset} \widetilde{\Sigma}^{\mathbf{X},T}$,
which follows from Corollary \ref{inclusions}
and the assumption $\sigma(\mathbf{X}) \overset{\mu}{\supset} \mathcal{A}$.
Moreover $(\widetilde{\mathcal{P}}_d^{(X_i)_{i=1}^n,T})_{d,n \in \mathbb{N}}$ is an increasing sequence
in $d$ and $N$ with respect to $\prec$, as it can be shown analogical to the proof of Lemma
\ref{increasing}. Thus, for $T$ ergodic
the assertion follows by Lemma \ref{walters}.
To show the non-ergodic case one can use the
the ergodic decomposition theorem as in the proof of
Theorem~\ref{main}. %changed 19.01.2015 12:06

It seems that the assumption $\sigma(\mathbf{X}) \overset{\mu}{\supset} \mathcal{A}$ in Theorem
\ref{main2} cannot be replaced by the assumption $\sigma((\mathbf{X} \circ T^{\circ t})_{t\in\mathbb{N}_0}) \overset{\mu}{\supset} \mathcal{A}$
in Theorem \ref{main}. At least, the argumentation of the proof of Corollary~\ref{inclusion} cannot be adapted.
Whereas
\begin{equation*}
\sigma(\mathbf{X}\circ T^{\circ t}) \overset{\mu}{\subset} \widetilde{\Sigma}^{\mathbf{X}\circ T^{\circ t},T}
\mbox{ for all }t \in \mathbb{N}_0
\end{equation*}
is true as \eqref{suneq} is, the analogue
\begin{equation*}
\widetilde{\mathcal{P}}_d^{X_i \circ T,T}\prec\widetilde{\mathcal{P}}_{d+1}^{X_i,T}\mbox{ for all }d \in \mathbb{N}\mbox{ and }i=1,2,\ldots,n
\end{equation*}
of \eqref{puneq} is false. Therefore the analogue
\begin{equation*}
\widetilde{\Sigma}^{X_i\circ T^{\circ t},T} \subset \widetilde{\Sigma}^{X_i,T}\mbox{ for all }i=1,2,\ldots,n\mbox{ and }t\in\mathbb{N}_0
\end{equation*}
of \eqref{sigmauneq}
is not guaranteed.
Let us give an example.

\begin{example}\label{example}
Let $\Omega=[0,1]$ and $T:\Omega\hookleftarrow$ be defined by
\begin{equation*}
T(\omega)
=
\left\{
\begin{array}{rl}
2\omega &\mbox{for }\omega\leq\frac{1}{2}\\
2-2\omega &\mbox{else}
\end{array}
\right. .
\end{equation*}
($T$ is the tent map preserving the equidistribution on $[0,1]$.)
Let
\begin{equation*}
Y=2\cdot {\bf 1}_{[0,\,1/3]}+3\cdot {\bf 1}_{]1/3,\,2/3]}+{\bf 1}_{]2/3,\,1]},
\end{equation*}
$\omega_1 = 1$ and $\omega_2 = \frac{5}{6}$. Then
\begin{equation*}
\begin{split}
(Y(T^{\circ t}(\omega_1))_{t\in {\mathbb N}_0}&=(1,2,2,2,2,2,\ldots),\\
(Y(T^{\circ t}(\omega_2))_{t\in {\mathbb N}_0}&=(1,2,3,3,3,3,\ldots).
\end{split}
\end{equation*}
It follows that $\omega_1$ and $\omega_2$ are separated by $\mathcal{P}^{Y \circ T,T}_{0,1}$ and
hence for all $\widetilde{\mathcal{P}}^{Y \circ T, T}_d$; $d \in \mathbb{N}$, but are not separated by
$\widetilde{\mathcal{P}}^{Y,T}_d$ for all $d \in \mathbb{N}$. Consequently,
$\widetilde{\mathcal{P}}^{Y \circ T, T}_d\hspace{-2mm}\not{\!\!\prec}\ \widetilde{\mathcal{P}}^{Y,T}_{d+l}$
for all $d \in \mathbb{N}$ and $l \in \mathbb{N}_0$.
\end{example}\vspace{3mm}

\textit{Other partitions.}
For a single observable $X$ on a measure-preserving dynamical system $(\Omega,\mathcal{A},\mu,T)$
and $s,t\in \mathbb{N}_0$ with $s<t$, let
\begin{equation*}
\begin{split}
\mathcal{Q}^{X,T}_{s,t}
=
\{&\{\omega \in \Omega \mid X ( T^{\circ s}(\omega)) > X ( T^{\circ t}(\omega))\} ,\\
&\{\omega \in \Omega \mid X ( T^{\circ s}(\omega)) \leq X ( T^{\circ t}(\omega))\}\}
\end{split}
\end{equation*}
and
\begin{equation*}
\begin{split}
\mathcal{R}^{X,T}_{s,t}
=
\{&\{\omega \in \Omega \mid X ( T^{\circ s}(\omega)) < X ( T^{\circ t}(\omega))\} ,\\
&\{\omega \in \Omega \mid X ( T^{\circ s}(\omega)) > X ( T^{\circ t}(\omega))\},\\
&\{\omega \in \Omega \mid X ( T^{\circ s}(\omega)) = X ( T^{\circ t}(\omega))\}\}.
\end{split}
\end{equation*}
Further, for observables $X_1,X_2,\ldots ,X_n$ on $(\Omega,\mathcal{A},\mu,T)$ and $d\in {\mathbb N}$, let
\begin{equation} \label{qpart}
\mathcal{Q}_d^{(X_i)_{i=1}^n,T}
=
\bigvee_{i=1}^n\ \bigvee_{0\leq s<t\leq d}\mathcal{Q}_{s,t}^{X_i,T}
\end{equation}
and
\begin{equation} \label{rpart}
\mathcal{R}_d^{(X_i)_{i=1}^n,T}
=
\bigvee_{i=1}^n\ \bigvee_{0\leq s<t\leq d}\mathcal{R}_{s,t}^{X_i,T}.
\end{equation}
(If one of the sets of the right hand side of \eqref{qpart} or \eqref{qpart} is empty, then it is not considered in order to have only nonempty sets.)
Then the following is valid:
\begin{corollary}
The statement of Theorem \ref{main} remains true when substituting\linebreak $\mathcal{P}_d^{(X_i)_{i=1}^n,T}$
by $\mathcal{Q}_d^{(X_i)_{i=1}^n,T}$ or $\mathcal{R}_d^{(X_i)_{i=1}^n,T}$.
\end{corollary}

\begin{proof}
Application of Theorem \ref{main} to $-{\bf X}=(-X_i)_{i\in {\mathbb N}}$ provides
\begin{equation*}
h_\mu^{\mathrm{KS}}(T)
=
\lim\limits_{d,n\to\infty} h_\mu(T,\mathcal{P}_d^{(-X_i)_{i=1}^n,T})=h_\mu(T,\mathcal{Q}_d^{(X_i)_{i=1}^n,T}).
\end{equation*}
Moreover, each $\mathcal{R}_d^{(X_i)_{i=1}^n,T}$ is finer than $\mathcal{P}_d^{(X_i)_{i=1}^n,T}$
implying
\begin{equation*}
h_\mu(T,\mathcal{R}_d^{(X_i)_{i=1}^n,T})\geq h_\mu(T,\mathcal{P}_d^{(X_i)_{i=1}^n,T}).
\end{equation*}
Therefore
\begin{equation*}
h_\mu^{\mathrm{KS}}(T)
\geq
\lim\limits_{d,n\to\infty} h_\mu(T,\mathcal{R}_d^{(X_i)_{i=1}^n,T})
\geq
\lim\limits_{d,n\to\infty} h_\mu(T,\mathcal{P}_d^{(X_i)_{i=1}^n,T})
\overset{\text{Theorem \ref{main}}}{=} h_\mu^{\mathrm{KS}}(T).
\end{equation*} %changed 19.01.2015 12:11
The existence of the limit
\begin{equation*}
\lim\limits_{d,n\to\infty} h_\mu(T,\mathcal{R}_d^{(X_i)_{i=1}^n,T})
\end{equation*}
and its coincidence with the corresponding supremum is obvious (compare discussion
for $\mathcal{P}_d^{(X_i)_{i=1}^n,T}$ in Section \ref{sec2}).
\end{proof}
Let us consider an order $\prec$ between observables $X,Y$ by $X\prec Y$ iff for all
$\omega_1,\omega_2\in\Omega$ the following holds (compare \cite{amigo_2012}):
\begin{equation*}
Y(\omega_1)\leq Y(\omega_2)\mbox{ implies }X(\omega_1)\leq X(\omega_2).
\end{equation*}
One easily shows the following:
\begin{lemma}
For $X\prec Y$ it holds $\mathcal{R}_d^{X,T}\prec \mathcal{R}_d^{Y,T}$.
\end{lemma}
Note that for $X\prec Y$ not generally $\mathcal{P}_d^{X,T}\prec \mathcal{P}_d^{Y,T}$ and
$\mathcal{Q}_d^{X,T}\prec \mathcal{Q}_d^{Y,T}$. After the following corollary being an immediate consequence of Theorem \ref{main}, we will
illustrate this point by an example.

\begin{corollary}\label{mainordered}
Let $(\Omega, \mathcal{A},\mu, T)$ be a measure-preserving dynamical system and
$\mathbf{X}=(X_i)_{i\in {\mathbb N}}$ be a sequence of observables with
$X_1\prec X_2\prec X_3\prec\ldots$ and
$\sigma(\{\mathbf{X} \circ T^{\circ t}\}_{t\in {\mathbb N}_0}) \overset{\mu}{\supset} \mathcal{A}$.
Assume that \eqref{ergodic} or \eqref{nonergodic} holds. Then
\begin{equation*}
h_\mu^{\mathrm{KS}}(T)
=
\lim\limits_{d,i\to\infty} h_\mu(T,\mathcal{R}_d^{X_i,T})=\sup_{d,i\in {\mathbb N}} h_\mu(T,\mathcal{R}_d^{X_i,T}).
\end{equation*}
\end{corollary}

\begin{example}
See Example \ref{example} and let
\begin{equation*}
X=2\cdot {\bf 1}_{[0,\,5/8]}+{\bf 1}_{]5/8,\,1]}
\end{equation*}
and
\begin{equation*}
Y=4\cdot {\bf 1}_{[0,\,1/8]\cup [3/8,5/8]}+3\cdot{\bf 1}_{]1/8,\,3/8[}+{\bf 1}_{]5/8,\,1]}.
\end{equation*} %changed 19.01.2015 12:14
Obviously, $X\prec Y$.
Let $\omega_1=\frac{1}{4}$ and $\omega_2=\frac{3}{4}$. Then
\begin{equation*}
\begin{split}
(X(T^{\circ t}(\omega_1))_{t\in {\mathbb N}_0}&=(2,2,1,2,2,2,2,2,2,\ldots),\\
(X(T^{\circ t}(\omega_2))_{t\in {\mathbb N}_0}&=(1,2,1,2,2,2,2,2,2,\ldots),\\
(Y(T^{\circ t}(\omega_1))_{t\in {\mathbb N}_0}&=(3,4,1,4,4,4,4,4,4,\ldots),
\end{split}
\end{equation*}
and
\begin{equation*}
(Y(T^{\circ t}(\omega_2))_{t\in {\mathbb N}_0}=(1,4,1,4,4,4,4,4,4,\ldots).
\end{equation*}
From this, on one hand it follows that $\omega_1$ and $\omega_2$ are separated by $\mathcal{P}^{X,T}_{0,1}$,
i.e.~lie in different elements of $\mathcal{P}^{X,T}_{0,1}$, hence are separated by
$\mathcal{P}^{X,T}_d$ for all $d\in {\mathbb N}$. On the other hand, this implies that $\omega_1$
and $\omega_2$ are not separated by $\mathcal{P}^{Y,T}_d$ for all $d\in {\mathbb N}$.

Therefore for no $d\in {\mathbb N}$ the partition $\mathcal{P}^{Y,T}_d$ is finer than $\mathcal{P}^{X,T}_d$.
The similar is true for $\mathcal{Q}^{Y,T}_d$ and $\mathcal{Q}^{X,T}_d$, since
$\mathcal{Q}^{Z,T}_d=\mathcal{P}^{-Z,T}_d$ for an observable $Z$ on $(\Omega,\mathcal{A},\mu,T)$.
\end{example}

\begin{remark}\label{refinement}
Each finite partition $\mathcal{C}=\{C_1,C_2,\ldots ,C_q\}\subset\mathcal{A}$; $q \in \mathbb{N}$ is generated by observables
of the form $X=\sum_{l=1}^q\alpha_l\cdot {\bf 1}_{C_l}$ in the sense that $C_l=X^{-1}(\alpha_l)$ for all $l=1,2,\ldots, q$, where $\alpha_l$; $l=1,2,\ldots , q$ are different real numbers.
If a partition $\mathcal{D}\subset\mathcal{A}$ is finer than $\mathcal{C}$, than it can be written
as
\begin{equation*}
\mathcal{D}=\bigcup_{l=1}^q\{D_j^{(l)}\mid j=1,2,\ldots, m_l\}
\end{equation*}
with $m_1,m_2,\ldots ,m_q\in {\mathbb N}$ and $C_l=\bigcup_{j=1}^{m_l}D_j^{(l)}$.

If $X=\sum_{l=1}^q\alpha_l\cdot {\bf 1}_{C_l}$ for different $\alpha_l\in {\mathbb N}$ and if $m>m_l$ for
all $l=1,2,\ldots ,q$, then for
\begin{equation*}
Y=\sum_{l=1}^q\sum_{j=1}^{m_l}(\alpha_l\cdot m+j)\,{\bf 1}_{D_j^{(l)}}
\end{equation*}
it holds $X\prec Y$. This shows that an increasing sequence $(\mathcal{C}_d)_{d\in {\mathbb N}}$ can be `generated'
by a sequence $(X_d)_{d\in {\mathbb N}}$ of observables with $X_1\prec X_2\prec X_3\prec\ldots$ .
\end{remark}

\section{Permutation entropy}\label{sec4}
The idea of considering dynamical systems from the ordinal viewpoint is strongly related to the invention of the permutation entropy,
which we want to discuss now. We first give a definition of it in our general framework:
\begin{definition}
Given a sequence $\mathbf{X}=(X_i)_{i \in \mathbb{N}}$ of observables on a measure-preserving dynamical system $(\Omega,\mathcal{A},\mu,T)$,
we define the \emph{permutation entropy} $h_\mu(T,{\bf X})$ with respect to $\mathbf{X}$ by
\begin{equation} \label{PE}
h^\ast_\mu(T,{\bf X})
= \lim_{n\to\infty}\limsup_{d\to\infty} \frac{1}{d}\,H_\mu(T,\mathcal{P}_d^{(X_i)_{i=1}^n,T}).
\end{equation}
\end{definition}
Originally, by Bandt et al.~in \cite{bandt_et_al_2002} the definition of permutation entropy was given directly for one-dimensional systems. In our framework,
this is $h^\ast_\mu(T,\mathrm{id})$ with $T$ being an interval map.\vspace{3mm}

{\textit{Permutation and Kolmogorov-Sinai entropy.}} One reason for investigating the permutation entropy is its close relationship to the well-established Kolmogorov-Sinai entropy
first observed by Bandt et al.~in \cite{bandt_et_al_2002}. In their seminal paper they have shown that both entropies
are coinciding for piecewise monotone interval maps $T$, i.e.~for selfmaps $T$ on intervals splitting into finitely many subintervals on which $T$ is
continuous and monotone.

Moreover, in the case that $\sigma((\mathbf{X} \circ T^{\circ t})_{t\in\mathbb{N}_0}) \overset{\mu}{\supset} \mathcal{A}$ and
that \eqref{ergodic} or \eqref{nonergodic} holds,
the Kolmogorov-Sinai entropy is not larger than permutation entropy.
It holds for finitely many observables
\begin{equation*}
 \lim\limits_{d\to\infty} h_\mu(T,\mathcal{P}_d^{(X_i)_{i=1}^n,T})
 \leq \limsup_{d\to\infty} \frac{1}{d}\,H_\mu(T,\mathcal{P}_d^{(X_i)_{i=1}^n,T})
 \text{ for all } n \in \mathbb{N}
\end{equation*}
(see Keller et al.~\cite[Corollary 3]{keller_et_al_2012}), hence the corresponding inequality for infinitely many ones follows by $n$ approaching to infinity. So let us
summarize:
\begin{corollary}
Let $(\Omega, \mathcal{A},\mu, T$) be a measure-preserving dynamical system
and $\mathbf{X} = (X_i)_{i \in \mathbb{N}}$ be a sequence of observables such that
$\sigma((\mathbf{X} \circ T^{\circ t})_{t\in\mathbb{N}_0}) \overset{\mu}{\supset} \mathcal{A}$.
Assume that \eqref{ergodic} or \eqref{nonergodic} holds.
Then
\begin{equation*}
h_\mu^{\mathrm{KS}}(T)
\leq
h^\ast_\mu(T,{\bf X}).
\end{equation*}
\end{corollary}\vspace{3mm}

{\textit{ The approach of Amig{\'o} et al.~\cite{amigo_2012,amigo_2005}.}} This approach to permutation entropy different to the original
is based on a refining sequence of finite partitions and is justified by the following statement due to Amig{\'o} et al.~\cite{amigo_2012,amigo_2005}.
We express the statement by finite-valued observables and refer here to Remark \ref{refinement}.
\begin{theorem}
For a measure-preserving dynamical system $(\Omega, \mathcal{A},\mu, T$) the following is valid:
\begin{enumerate}
\item[(i)] If $X$ is a finitely-valued observable, and $\mathcal{P}$ the
finite partition generated by $X$, then
\begin{equation*}
h_\mu(T,\mathcal{P})=h^\ast_\mu(T, X).
\end{equation*}
\item[(ii)] If $(X_i)_{i\in {\mathbb N}}$ is a sequence of finitely-valued observables with $X_1\prec X_2\prec X_3\prec\ldots$ and the corresponding sequence
of finite partitions generates $\mathcal{A}$, then
\begin{equation}\label{modified}
h_{\mu}^{\mathrm{KS}}(T)
=\lim\limits_{i\to\infty} h^\ast_\mu(T,X_i).
\end{equation}
\end{enumerate}
\end{theorem}
One immediately sees that by Lemma \ref{walters} assertion
(ii) follows directly from statement (i). Amig{\'o} et al.~took the right hand side of \eqref{modified}
as their modified concept of permutation entropy before showing its equality to Kolmogorov-Sinai entropy.

We want to finish this section by stating the following general problem, which is interesting on the different levels from the original one-dimensional definition of permutation entropy
to the generalization for finitely or infinitely many observables.
\begin{problem}
Are the Kolmogorov-Sinai entropy and the permutation entropy coinciding and, if not, under which assumptions?
\end{problem}
Note that the pure combinatorial part of the problem is relatively well understood (see Unakafova et al.~\cite{unakafova_et_al_2013}, Keller et al.~\cite{keller_et_al_2012}).

\section{Ordinal time series analysis}\label{sec5}

Ever since the idea of Bandt and Pompe~\cite{bandt_pompe_2002} to consider the rank order
of consecutive values of a time series instead of the values themselves, the ordinal approach attracts increasing attention and is applied in many
scientific fields, for example in biomedical research, engineering and econophysics (see Amig{\'o} et al.~\cite{amigo_eta_al_2014,amigo_eta_al_2013}, Zanin et al.~\cite{zanin_et_al_2012}
and the references given there).

The reason is that the ordinal viewpoint brings with it many advantages especially for measuring complexity, such as
robustness against small noise, simplicity of application and interpretation, and low computational costs.
As mentioned, the determination of Kolmogorov-Sinai entropy is usually not easy, our discussion above, however,
suggests that the ordinal approach can be used as a framework
for estimating the Kolmogorov-Sinai entropy of dynamical systems and suchlike from real world data.

In the following we consider the theory developed in the previous sections in an applied context and discuss
the pro and cons of using this approach in view of studying long and complex time series. A detailed exposition of this ordinal pattern approach
is provided in Keller et al.~\cite{keller_et_al_2007}.
\vspace{3mm}

\textit{Ordinal patterns.}
The task of gaining information about an underlying system via measurements is
a common everyday problem. As already mentioned, this issue is increasingly addressed
by using information lying in the ordinal structure of a system or a time series obtained from it.
This leads to considering the
up and downs in a time series, which can be described via so-called ordinal patterns.

\begin{definition}
For $d\in \mathbb{N}$ denote the \emph{set of permutations} of $\{0,1,\dots,d\}$
by $\Pi_d$.
We say that a real vector $(x_s)_{s=0}^d$
has \emph{ordinal pattern} $\boldsymbol{\pi} = (\pi_0,\pi_1,\dots,\pi_d) \in \Pi_d$
\emph{of order $d$} if
\begin{equation*}
x_{\pi_0}
\geq
x_{\pi_1}
\geq
\dots
\geq
x_{\pi_{d-1}}
\geq
x_{\pi_{d}}
\end{equation*}
and
\begin{equation}\label{equality}
\pi_{u-1}>\pi_u\mbox{ if }x_{\pi_{u-1}} = x_{\pi_u}\mbox{ for any }u \in \{1,2,\dots,d\}.
\end{equation}
Given a time series $(x_t)_{t\in {\mathbb N}_0}$, the \emph{ordinal pattern of order $d$ at time $t$} is defined
as that of $(x_{t+s})_{s=0}^d$ and denoted by $\boldsymbol{\pi}_t$.
\end{definition}

\begin{example}
In Figure~\ref{figure} we consider a time series of $50$ data points where
exemplary the ordinal pattern
$\boldsymbol{\pi}_{10}=(0,5,3,4,6,1,2) \in \Pi_6$
is emphasized,
which corresponds to the order relation of the six successive values at $t=10$, that is
\begin{equation*}
x_{t} > x_{t+5} > x_{t+3} > x_{t+4} > x_{t+6} > x_{t+1} > x_{t+2};\ t= 10.
\end{equation*}

\begin{figure}[htp]
\begin{tikzpicture}
\begin{axis}[
          width=1\textwidth,
          height=0.2\textheight,
          ymin=0,
          xmin=0,
          xmax=50,
          ytick=\empty,
	  xtick={5,10,15,20,25,30,35,40,45},
	  xlabel=$t$,
	  %ylabel=$x_t$,
	  %every axis x label/.style={at={(current axis.right of origin)},anchor=north west}
	  every axis x label/.style={at={(current axis.right of origin)},anchor=north},
	  %every axis y label/.style={at={(current axis.above origin)},anchor=east}
          ]
         \addplot [color=gray, mark=*, mark options={solid},line width = 1.0 pt, mark size = 2.0 pt] table {datap.data};
         \addplot [color=black, mark=*, mark options={solid},line width = 1.0 pt, mark size = 2.0 pt] table {data1.data};
         \draw [color=gray, dashed] (axis cs:10,69) -- node[left]{} (axis cs:10,170);
         \draw [color=gray, dashed] (axis cs:11,55) -- node[left]{} (axis cs:11,120);
         \draw [color=gray, dashed] (axis cs:12,43) -- node[left]{} (axis cs:12,110);
          \draw [color=gray, dashed] (axis cs:13,66) -- node[left]{} (axis cs:13,150);
          \draw [color=gray, dashed] (axis cs:14,65) -- node[left]{} (axis cs:14,140);
         \draw [color=gray, dashed] (axis cs:15,68) -- node[left]{} (axis cs:15,160);
         \draw [color=gray, dashed] (axis cs:16,64) -- node[left]{} (axis cs:16,130);
\end{axis}
\end{tikzpicture}
\caption{Illustration of an ordinal pattern of order $d=6$ assigned to six successive values (plotted in the vertical direction) of a time series of $50$ data points.}\label{figure}
\end{figure}
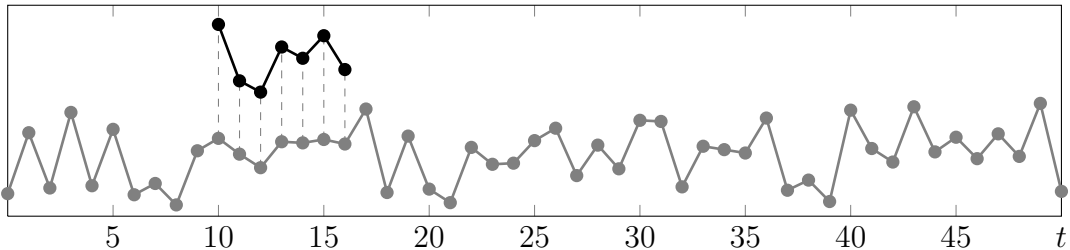
\end{example}

It is easily seen that, following the framework given in Section \ref{sec1},
two states $\omega_1 \in \Omega$ and $\omega_2\in\Omega$ belong to the same part of some ordinal partition $\mathcal{P}_d^{(X_i)_{i=1}^n,T}$
iff the ordinal patterns of the vectors
\begin{equation*}
(X_i(\omega_1),X_i(T(\omega_1)),\ldots ,X(T^{\circ d}(\omega_2))) \text{ and } (X_i(\omega_2),X_i(T(\omega_2)),\ldots ,\linebreak X(T^{\circ d}(\omega_2)))
\end{equation*}
coincide.
Clearly, the other
previous considered partitions
(see Equations \eqref{less},\eqref{qpart} and \eqref{rpart}),
despite some adjustments in terms of equality, can be coherent
assimilated to this ordinal approach by redefining ordinal patterns in terms of the equality of values.
The setting \eqref{equality} is here in some sense arbitrary,
however, the proposed definition of ordinal patterns has established itself.
We will use it in the following to demonstrate
how the previous covered theory provides
interesting and promising tools for extracting the information saved in an ordinal
pattern sequence or suchlike, for example, by estimating the permutation entropy (see Equation \eqref{PE})
or by approximating the Kolmogorov-Sinai entropy.

In order to utilize ordinal patterns for the analysis of a system,
sequential data $(x_t)_{t\in {\mathbb N}}$ obtained from a given measurement are transformed into a series $(\boldsymbol{\pi}_t)_{t\in {\mathbb N}_0}$ of ordinal patterns.
Distributions of ordinal patterns obtained from this approach are the central objects of exploration.

Note that ordinal patterns do not provide a symbolic representation as it is usually considered, since partitions of the state
space are not given a priori, but are created on the basis of the given dynamics. However, the ordinal patterns as `symbols'
are very simple objects being directly obtained from the orbits of the system and containing intrinsic causal information.
For the relationship of symbolic dynamics and representations and ordinal time series analysis see Amig{\'o} et al.~\cite{amigo_eta_al_2014}.

For simplicity, we now restrict our exposition to the one-dimensional case with only one measurement. What we
have in mind is a measure preserving dynamical system $(\Omega,\mathcal{A},\mu,T)$, where $\Omega$ is a Borel subset of ${\mathbb R}$, acting as the model of a system, with
a single observable $X$ being the identity map. The extension of the ideas to the general case is obvious.\vspace{3mm}

\textit{Estimation of ordinal quantities.}
The naive and mainly used estimator of ordinal pattern probabilities, so of the probability of the ordinal partition parts, is the relative
frequency of ordinal patterns in an orbit of some length. For some $t,d\in {\mathbb N}$, some ordinal pattern $\boldsymbol{\pi}$ of order $d$ and some $\omega\in\Omega$ the estimation
is given by the number
\begin{eqnarray*}
\hat{p}_{\boldsymbol{\pi}}=\frac{1}{t-d+1}\#\{s\in\{0,1,\ldots ,t-d\}\mid (X(T^{\circ s}(\omega)),X(T^{\circ s+1}(\omega)),\\\ldots
,X(T^{\circ s+d}(\omega)))\mbox{ has ordinal pattern }\boldsymbol{\pi}\}.
\end{eqnarray*}
Here $t+1$ is the length of the considered orbit of $\omega$. Clearly, the estimation only makes sense in the ergodic case. Then, by
Birkhoff's ergodic theorem, the corresponding estimator is consistent.

If in the ergodic case all $\hat{p}_{\boldsymbol{\pi}}$; $\pi \in \Pi_d$ are determined, it follows
immediately that in the simple case considered
a reasonable estimator for \eqref{PE} is given by the \emph{empirical permutation entropy of order} $d \in \mathbb{N}$:
\begin{equation*}
\hat{h}^\ast_\mu(T,X)= - \frac{1}{d} \sum_{\pi \in \Pi_d} \hat{p}_{\pi} \ln \hat{p}_{\pi}.
\end{equation*}
It gives furthermore also some information on the Kolmogorov-Sinai entropy.
\vspace{3mm}

\textit{Assets and drawbacks.}
Irrespective of the considered ordinal partition, the ordinal approach
brings along some practical advantages and disadvantages. Note that most
difficulties to overcome are common to any sort of time series analysis.

Considering the order
relation between the values of a time series, small inaccuracies in measurements
(e.g. errors between the state of a system and its observed value) are mostly negligible.
Hence, the methods considered are relatively robust towards calibration differences of measuring instruments.
Furthermore, the ordinal approach is easily interpretable and there already exist
efficient methods to perform an ordinal time series analysis in real time. For a deeper
discussion we refer to Riedl et al.~\cite{riedl_et_al_2013} as well as Unakafova and Keller \cite{unakafova_keller_2013}. Last but not least,
a foreknowledge of the data range when analyzing data is usually not necessary.

In contrast, the ordinal analysis of time series can be rather poor if the underlying system is so complex
that such a large value $d$ is needed that the computational capacity is insufficient. If, for example, the permutation
entropy of a dynamical system is very large, its estimation by the empirical permutation entropy is problematic.
Note that generally also for simple systems the convergency of empirical permutation entropies of order $d$ to the permutation entropy
can be rather slow, which is the reason for considering a conditional adaption of the permutation entropy (see Unakafov and Keller \cite{unakafov_keller_2014}).

In addition, the choice of a
suitable order $d$ with respect to the length
of the original time series is affected by common problems.
Large values of $d$ are needed to evaluate encapsulated information as accurate as possible but a
large $d$ grants $(d+1)!$ possible ordinal patterns which have to be considered if nothing is known about
the original time series. If one chooses an overlarge $d$ relative to the length of a time series, it can happen that not all ordinal patterns which are substantial
for describing the underlying dynamics are observed in the ordinal pattern distribution or suchlike. This is known
as \emph{undersampling}.

Moreover, ordinal time series analysis can lead to
an arbitrary poor approximation of the Kolmogorov-Sinai entropy or poor representation of
the underlying dynamics by the statistics,
especially while working on wrong assumptions, e.g.~a given system fails to be ergodic
or the chosen observables cause information loss while measuring.
The next section alludes to the latter problem.

\section{Algebra reconstruction dimension}\label{sec6}
Theorems \ref{main} claims that the Kolmogo\-rov-Sinai entropy of $T$ can be computed provided that we have sufficiently many observables ``generating'' $\mathcal{A}$ up to $\mu$-measure zero.
Essential for applications, the natural question arises how we can decrease the number of observables as much as possible.
%A natural question which is essential for applications is how can we decrease the number of observables as much as possible?
In this section we briefly review the known results in this direction.

\textit{Only one observable.} The following example shows that \emph{theoretically} in most real cases we can find only one such observable.

\begin{example}\label{exmp:std_spaces}
Let $I=[0,1]$, $Z$ be a separable complete metric space (such spaces are called \emph{Polish}), $\Omega\subset Z$ be its uncountable Borel subset, and $\mathcal{A} := \mathcal{B}(\Omega)$ be the Borel $\sigma$-algebra of $\Omega$.
Then the pair $(\Omega,\mathcal{B}(\Omega))$ is called a \emph{standard Borel space}.
It is well known, e.g.~see Kechris \cite[Proposition 12.1]{kechris_1995}, that then there exists a measurable isomorphism of $(\Omega,\mathcal{B}(\Omega))$ onto the space $\bigl(I, \mathcal{B}(I)\bigr)$, that is a bijection $\mathbf{X}:\Omega\to I$ such that
$\mathbf{X}^{-1}(\mathcal{B}(I)) = \mathcal{B}(\Omega)$.

Let $\mu$ be a measure on $(\Omega,\mathcal{B}(\Omega)$ and $T:\Omega\to\Omega$ be any $\mu$-preserving map.
Then
\begin{equation*}%\label{equ:exmp:sigmaTX_BO}
\mathcal{B}(\Omega)
\supset
\sigma((\mathbf{X}\circ T^{\circ t})_{t\in\mathbb{N}_0})
\supset
\sigma(\mathbf{X})
=
\mathbf{X}^{-1}(\mathcal{B}(I))
=
\mathcal{B}(\Omega),
\end{equation*}
that is
$\sigma((\mathbf{X}\circ T^{\circ t})_{t\in\mathbb{N}_0}) = \mathcal{B}(\Omega)$.
Moreover, as every separable metric space $Z$ can be embedded into a Hilbert cube being a compact space, compare to Hurewicz and Wallmann \cite[Chapter V, \S5, Theorem V4]{hurewicz_wallman_1941},
we see that condition \eqref{nonergodic} holds for $\Omega \subset Z$ as well, and therefore by Theorem \ref{main} the Kolmogorov-Sinai entropy $h^{\mathrm{KS}}(T)$ of $T$
can be computed via the formula \eqref{equ:hKS_lim_perm}.
\end{example}

Notice that the function $\mathbf{X}:\Omega\to[0,1] \subset \mathbb{R}$ from Example \ref{exmp:std_spaces} is not in general continuous and its explicit construction is very complicated.
Therefore it is not useful for real applications.
This leads to the following notion.

\begin{definition}\label{def:ard}
Let $(\Omega, \mathcal{B}(\Omega))$ be a standard Borel space with measure $\mu$ on $\mathcal{B}(\Omega)$, and $T:\Omega\to\Omega$ be a $\mathcal{B}(\Omega)$-$\mathcal{B}(\Omega)$-measurable map.
By the \emph{algebra reconstruction dimension} of $T$ with respect to $\mu$ we will mean the minimal integer number $n\geq1$ such that there exists a \emph{continuous} map $\mathbf{X}:\Omega \to\mathbb{R}^n$ satisfying
\begin{equation}\label{equ:def:ard}
\sigma((\mathbf{X}\circ T^{\circ t})_{t\in\mathbb{N}_0}) \overset{\mu}{\supset} \mathcal{B}(\Omega).
\end{equation}
This number will be denoted by $\mathrm{ard}_{\mu}(T)$.
If such $n$ does not exist, then we will assume that $\mathrm{ard}_{\mu}(T)=\infty$.
\end{definition}
Thus $\mathrm{ard}_\mu (T)$ is the minimal number of \emph{continuous} observables needed to approximate the Kolmogorov-Sinai entropy via \eqref{equ:hKS_lim_perm}.

Given a map $T:\Omega\to \Omega$, a map $\mathbf{X}:\Omega\to\mathbb{R}^n$ and $t\in\mathbb{N}$ one can define the following \emph{$t$-reconstruction} map
\begin{equation*}
\Lambda_{\mathbf{X}, T, t}
=
\bigl(\mathbf{X},  \mathbf{X}\circ T,  \ldots,  \mathbf{X}\circ T^{\circ t-1}\bigr) : \Omega\to\mathbb{R}^{nt}
\end{equation*}
and an \emph{$\infty$-reconstruction} map
\begin{equation*}
\Lambda_{\mathbf{X}, T, \infty}
=
\bigl(\mathbf{X},  \mathbf{X}\circ T,  \mathbf{X}\circ T^{\circ 2},  \ldots\bigr) : \Omega\to\mathbb{R}^{\infty}.
\end{equation*}
Evidently, $\Lambda_{\mathbf{X},T,1} = \mathbf{X}$,
\begin{equation*}
\sigma((\mathbf{X}\circ T^{\circ s})_{s=0}^{t-1}) = \sigma(\Lambda_{\mathbf{X}, T, t}),
\end{equation*}
and
\begin{equation*}
\sigma(\mathbf{X}) \ \subset \ \sigma(\Lambda_{\mathbf{X}, T, t}) \
\subset
\ \sigma(\Lambda_{\mathbf{X}, T, t+1}) \
\subset
\ \sigma(\Lambda_{\mathbf{X}, T, \infty}); \ t\in\mathbb{N}.\end{equation*}
In particular, \eqref{equ:def:ard} can be reformulated as follows:
\begin{equation}\label{equ:sigmaLambda_in_BOmega}
\sigma(\Lambda_{\mathbf{X}, T, \infty}) \ \overset{\mu}{\supset} \ \mathcal{B}(\Omega).
\end{equation}

Before discussing $\mathrm{ard}_{\mu}(T)$ we will present an example for the existence
of one separating observable, that is $\mathbf{X}:\Omega\to\mathbb{R}$ satisfying \eqref{equ:sigmaLambda_in_BOmega},
and therefore allowing to approximate the Kolmogorov-Sinai entropy by formula \eqref{equ:hKS_lim_perm}, see Theorem \ref{th:gen_partitions} below.
However, now this observable is ``discrete'', i.e.~it takes at most countable many values.

\begin{definition}\label{def:gen_partition}
Let $(\Omega,\mathcal{A},\mu,T)$ be a measure-preserving dynamical system.
An at most countable partition $\mathcal{C} = \{C_l\}_{l=1}^q \subset \mathcal{A}$ of $\Omega$ for some $q\in\mathbb{N}\cup\{\infty\}$, is called \emph{generating} with respect to $T$, if
\begin{equation*}
\sigma ((T^{-t} \mathcal{C})_{t\in {\mathbb N}_0})\overset{\mu}{=}  \mathcal{A},
\end{equation*}
where $T^{-t} \mathcal{C} = \{ (T^{\circ t})^{-1}C_l\}_{l=1}^q$.
\end{definition}
The following lemma is evident.
\begin{lemma}\label{lm:func_for_gen_part}
Suppose a measure-preserving dynamical system $(\Omega,\mathcal{A},\mu,T)$ has a generating partition $\mathcal{C}=\{C_l\}_{l=1}^q$; $q\in\mathbb{N}\cup\{\infty\}$.
Define a function $\mathbf{X}:\Omega\to\mathbb{R}$ by $\mathbf{X}=\sum_{l=1}^ql\cdot {\bf 1}_{C_l}$ (compare Remark \ref{refinement}).
Then $\sigma(\mathbf{X}) = \sigma(\mathcal{C})$, whence
\begin{equation*}
\sigma(\Lambda_{\mathbf{X}, T, \infty})
=
\sigma\bigl((\mathbf{X}\circ T^{\circ t})_{t\in\mathbb{N}_0}\bigr)
=
\sigma ((T^{-t} \mathcal{C})_{t\in {\mathbb N}_0})
\overset{\mu}{=}
\mathcal{A}.
\end{equation*}
%Suppose Let Given a generating partition $\{A_i\}_{i\in\mathbb{N}}$  $(\Omega,\mathcal{A},\mu,T)$ be a measure-preserving dynamical system.
\end{lemma}

In general, a $\mu$-preserving map does not have a generating partition.
Nevertheless, for non-singular ergodic automorphisms of standard probability spaces such partitions do exist, what we discuss now.
First we recall necessary definitions.

%\subsection{Ergodic maps}
Let $(\Omega, \mathcal{A}, \mu)$ be a probability space.
The measure $\mu$ is called \emph{complete} if
for any subset $A\in\mathcal{A}$ with $\mu(A)=0$ every its subset $B$ also belongs to $\mathcal{A}$.

A countable family of sets $\{A_l\}_{l\in\mathbb{N}} \subset \mathcal{A}$ is called a \emph{complete basis} of $(\Omega, \mathcal{A}, \mu)$ if
\begin{enumerate}
\item[\rm(a)]
for each $A\in\mathcal{A}$ there exists a $B \in \sigma(\{A_i\}_{l=1}^\infty)$ with $A \subset B$ and $\mu(B\setminus A) = 0$;
\item[\rm(b)]
for any $\omega_1,\omega_2\in\Omega$ there exists an $l\in\mathbb{N}$ such that $\omega_1\in A_l$ and $\omega_2\in\Omega\setminus A_l$;
\item[\rm(c)]
each intersection $\bigcap_{l\in\mathbb{N}} B_l$, where every $B_l$ is either $A_l$ or $\Omega\setminus A_l$, is non-empty.
%  x,y\in\Omega$ there exist $i\in\mathbb{N}$ such that $x\in A_i$ and $y\in\Omega\setminus A_i$;
\end{enumerate}

A probability space $(\Omega, \mathcal{A}, \mu)$ is called \emph{standard} if it has a complete basis and $\mu$ is complete.

It has been proved by Rohlin \cite{rohlin_1961} that every standard probability space with non-atomic measure is isomorphic with the probability space $(I, \mathcal{B}(I), \lambda)$, where $\lambda$ is the Lebesgue measure on $I$.

Recall also that a one-to-one transformation $T:\Omega\to\Omega$ is non-singular with respect to a measure $\mu$ if it is bi-measurable, i.e.~$T^{-1}\mathcal{A}=\mathcal{A}$ and $T\mathcal{A}=\mathcal{A}$, and $\mu(A)=0$ if and only if $\mu(T(A))=0$ for all $A\in \mathcal{A}$.

The following theorem is a consequence of results by Rohlin \cite{rohlin_1961}, Parry \cite{parry_1966} and Krieger \cite{krieger_1970} about the existence of countable and finite generating partitions of ergodic maps.

\begin{theorem}\cite{rohlin_1961,parry_1966,krieger_1970}\label{th:gen_partitions}
Let $(\Omega, \mathcal{B}(\Omega), \mu)$ be a standard probability space, and $T:\Omega\to\Omega$ be a non-singular ergodic $\mu$-preserving map.
Then $\Omega$ has a countable generating partition with respect to $T$.
Hence there is a discrete measurable function $\mathbf{X}:\Omega\to\mathbb{R}$ taking at most countable distinct values and satisfying \eqref{equ:sigmaLambda_in_BOmega}.

Moreover, if $h^{\mathrm{KS}}(T)<\infty$, then $T$ admits a \emph{finite} generating partition, and so $\mathbf{X}$ can be assumed to take only finitely many distinct values.
\end{theorem}

\textit{The continuous case.} Notice that the function $\mathbf{X}$ from Theorem \ref{th:gen_partitions} is slightly better than the one from Example \ref{exmp:std_spaces}, as it takes a discrete set of values mutually distinct for distinct elements of the
generating partition $\mathcal{C}$.
Nevertheless, it is hard to construct as it requires to know a generating partition for $T$, and so it is not useful for application as well.

Now we will consider the opposite situation when almost any continuous map $\mathbf{X}:\Omega\to\mathbb{R}^n$ satisfies \eqref{equ:sigmaLambda_in_BOmega}.

\begin{lemma}\label{lm:emb_th_ard_2k1}
Let $\Omega$ be a Polish space admitting an embedding $\mathbf{X}:\Omega\to\mathbb{R}^{n}$.
Then for any measure $\mu$ on $\mathcal{B}(\Omega)$ and any $\mu$-preserving map $T$, we have that $\mathrm{ard}_{\mu}(T)\leq n$.
In particular, if $\dim\Omega=k$; $k \in \mathbb{N}$, then $\mathrm{ard}_{\mu}(T)\leq 2k+1$.
\end{lemma}
\begin{proof}
Since $\mathbf{X}$ is an embedding, we obtain that $\sigma(\mathbf{X}) = \mathbf{X}^{-1}(\mathcal{B}(\mathbb{R}^n)) = \mathcal{B}(\Omega)$, whence
$\sigma(\Lambda_{\mathbf{X}, T, \infty}) = \mathcal{B}(\Omega)$ as well.

The second statement follows from the well known fact that every $k$-dimensional separable metric space $\Omega$ can be embedded into $\mathbb{R}^{2k+1}$, \cite[Chapter V, \S4, Theorem V3]{hurewicz_wallman_1941}.
Moreover, by the same theorem the set of embeddings $\mathrm{Emb}(\Omega, \mathbb{R}^{2k+1})$ is residual (and, in particular, dense) in the space $C(\Omega,\mathbb{R}^{2k+1})$ of all continuous maps.
Therefore \emph{almost every} family of $2k+1$ \emph{continuous} observables will allow to approximate the Kolmogorov-Sinai entropy of $T$.
\end{proof}

The next statement is a slight generalization of Theorem 2.2 from Keller \cite{keller_2012}.

\begin{theorem}\label{th:keller_diff}
Let $\Omega$ be a smooth manifold and $\mathcal{D}(\Omega)$ be the group of its $C^{\infty}$ diffeomorphisms.
Then there exists a residual subset $\mathcal{W}$ of $\mathcal{D}(\Omega)$ such that $\mathrm{ard}_{\mu}(T)=1$ for each $T\in\mathcal{W}$ and any measure $\mu$ preserved by $T$.
\end{theorem}
\begin{proof}
Let $\dim\Omega=k$.
For each $n\in\mathbb{N}$ let %denote by $\mathcal{V}_{n}(\Omega)$ a subset
\begin{equation*}
\mathcal{E}_{n}
=
\{ (\mathbf{X}, T) \in C^{\infty}(\Omega,\mathbb{R})\times\mathcal{D}(\Omega) \mid \Lambda_{\mathbf{X},T,n}:\Omega\to\mathbb{R}^{n} \ \text{is an embedding}\}.
\end{equation*}
Thus if $(\mathbf{X}, T)\in \mathcal{E}_{n}$, then $\mathrm{ard}_{\mu}(T) = 1$.

It is proved by Takens \cite{takens_81} that if $n\geq 2 k + 1$, then $\mathcal{E}_{n}$ is residual (and in particular non-empty and everywhere dense) in $C^{\infty}(\Omega,\mathbb{R})\times\mathcal{D}(\Omega)$.
Thus we have that $\mathcal{E}_{2k+1} = \bigcap_{l=1}^{\infty} U_l$, where each $U_i$ is open and everywhere dense in the space $C^{\infty}(\Omega,\mathbb{R})\times\mathcal{D}(\Omega)$.
Let $p:C^{\infty}(\Omega,\mathbb{R})\times\mathcal{D}(\Omega) \to \mathcal{D}(\Omega)$ be the natural projection, i.e.~$p(\mathbf{X},T) = T$.
It is a standard fact from general topology that $p$ is an open map, whence
\begin{equation*}
\mathcal{W} = p\bigl(\mathcal{E}_{2k+1}\bigr)
=
\bigcap_{l=1}^{\infty} p(U_l)
\end{equation*}
is a residual subset of $\mathcal{D}(\Omega)$.
Then $\mathrm{ard}_{\mu}(T)=1$ for each $T\in \mathcal{W}$ and any measure $\mu$ preserved by $T$.
\end{proof}

Notice that the latter result does not guarantee that for \textit{any} measure $\mu$ on $\mathcal{B}(\Omega)$ preserved by some diffeomorphism $T$ there exists some other $\mu$-preserving diffeomorphism $T'$ with $\mathrm{ard}_{\mu}(T')=1$.

The following notion allows to decrease the dimension $2k+1$ in Lemma \ref{lm:emb_th_ard_2k1} by putting some restrictions on $\mu$.

\begin{definition}
Let $\mathbf{X}:\Omega\to R$ be a continuous map between topological spaces.
Then the following subset of $\Omega$
\begin{equation*}
N_{\mathbf{X}} = \{ \omega\in \Omega \mid \mathbf{X}^{-1}( \mathbf{X}(\omega)) \not= \{\omega\} \}
\end{equation*}
will be called the \emph{set of non-injectivity} of $\mathbf{X}$.
\end{definition}
\begin{lemma}{\rm (Antoniouk et al.~\cite[Theorem 4.2]{antoniouk_et_al_2013})}
Let $\mathbf{X}:\Omega\to R$ be a continuous map between Polish spaces and $\mu$ be a measure on $\mathcal{B}(\Omega)$.
Suppose there exists a Borel subset $D$ such that $N_{\mathbf{X}} \subset D$ and $\mu(D)=0$.
Then $\sigma(\mathbf{X}) \overset{\mu}{=} \mathcal{B}(\Omega)$.
\end{lemma}

Let $\Omega$ be a smooth manifold of dimension $k$.
Say that a subset $Q \subset \Omega$ \emph{has Lebesgue measure zero}, if for any local chart $\phi:\Omega \supset U \to \mathbb{R}^k$ in $\Omega$ the set $\phi(Q\cap U)$ has Lebesgue measure zero in $\mathbb{R}^k$.
Notice that there is no natural definition of a set of \emph{fixed positive Lebesgue measure}.

A measure $\mu$ on $\mathcal{B}(\Omega)$ will be said \emph{Lebesgue absolutely continuous} if $\mu(Q)=0$ for each subset $Q\subset\Omega$ of measure zero.

\begin{theorem}\label{th:noninj_set}{\rm (Antoniouk et al.~\cite[Theorem 2.13]{antoniouk_et_al_2013})}
Let $\Omega$ be a smooth manifold of dimension $k$ and $\mu$ be a Lebesgue absolutely continuous measure on $\mathcal{B}(\Omega)$.
For each $n\in\mathbb{N}$ let
\begin{equation*}
\mathcal{V}_{n} = \{ \mathbf{X} \in C^{\infty}(\Omega,\mathbb{R}^n) \mid N_{\mathbf{X}}\in\mathcal{B}(\Omega), \ \mu(N_{\mathbf{X}})=0\}.
\end{equation*}
If $n>k$, then $\mathcal{V}_{n}$ is residual in $C^{\infty}(\Omega,\mathbb{R}^n)$.
Hence $\mathrm{ard}_{\mu}(T) \leq k+1$ for any (not necessarily continuous) $\mu$-preserving map $T:\Omega\to\Omega$.
\end{theorem}

\subsection{Comparison of results}
It is convenient to compare these results in the following table, where it is assumed that $\Omega$ is a Polish space of dimension $k$.\vspace{3mm}

\begin{center}
\begin{tabular}{|p{1.6cm}|p{2.3cm}|p{3cm}|c|c|} \hline
\centering$\Omega$ & \centering $\mu$ & \centering $T$ & \centering $\mathrm{ard}_{\mu}(T)$ & Statement \\ \hline
\centering Borel space        &
\centering any measure        &
\centering any $\mu$-preserving measurable map &
$\leq 2k+1$             &
Lemma \ref{lm:emb_th_ard_2k1} \\ \hline
\centering Smooth manifold    &
\centering Lebesgue absolutely continuous &
\centering any $\mu$-preserving measurable map &
$\leq k+1$  &
Theorem \ref{th:noninj_set} \\ \hline
\centering Smooth manifold    &
\centering any measure preserved by $T$ &
\centering generic diffeomorphism &
$1$  &
Theorem \ref{th:keller_diff} \\ \hline
\end{tabular}
\end{center}

\medskip
% The data information below will be filled by AIMS editorial staff
Received xxxx 20xx; revised xxxx 20xx.
\medskip

\end{document}